\documentclass{amsart}
\usepackage{amssymb}
\usepackage{amsmath,amscd}
\usepackage{color}
\usepackage{epsf,epsfig}
\usepackage{graphicx}
\input xyv2
\input xy

\theoremstyle{plain}
 \newtheorem{thm}{Theorem}[section]
 \newtheorem{cor}[thm]{Corollary}
 \newtheorem{lem}{Lemma}[section]
 \newtheorem{prop}{Proposition}[section]
 \theoremstyle{definition}
 \newtheorem{defn}{Definition}[section]
 
 \theoremstyle{remark}
 \newtheorem{rem}{Remark}
 \newtheorem{example}{Example}[section]
 \newtheorem*{acknow}{Acknowledgements}
 \numberwithin{equation}{section}

  \newcommand{\Mob}{\mathcal{M}}

 \newcommand{\R}{\mathbb{R}}

 \newcommand{\To}{\longrightarrow}

 \newcommand{\Z}{\mathbb{Z}}

\begin{document}

\title[M\"obius transform,  moment-angle complexes and H--C conjecture]
 {M\"obius transform,  moment-angle Complexes and Halperin--Carlsson conjecture}

\author{Xiangyu CAO and Zhi L\"{U}}
\address{School of Mathematical Sciences, Fudan University, Shanghai, 200433, People's Republic of China.}
\thanks{Supported by grants from FDUROP (No. 080705) and
 NSFC (No. 10671034, No. J0730103)}
\email{xiangyu.cao08@gmail.com}
\address{Institute of Mathematics, School of Mathematical Science, Fudan University, Shanghai,
200433, People's Republic of China.}
 \email{zlu@fudan.edu.cn}

\begin{abstract}
In this paper,  we give an algebra--combinatorics formula of the
M\"obius transform for  an abstract simplicial complex $K$ on
$[m]=\{1, ..., m\}$ in terms of the Betti numbers of the
Stanley--Reisner face ring of $K$. Furthermore, we  employ a way of
compressing $K$ to estimate the lower bound of the sum of those
Betti numbers by using this formula. As an application, associating
with the moment-angle complex $\mathcal{Z}_K$ (resp. real
moment-angle complex ${\Bbb R}\mathcal{Z}_K$) of $K$, we show that
the Halperin--Carlsson conjecture holds for $\mathcal{Z}_K$ (resp.
${\Bbb R}\mathcal{Z}_K$) under the restriction of the natural
$T^m$-action on $\mathcal{Z}_K$ (resp. $({\Bbb Z}_2)^m$-action
 on ${\Bbb R}\mathcal{Z}_K$).
\end{abstract}

\maketitle

\section{Introduction}

Throughout this paper, assume that  $m$ is a positive integer and
$[m]=\{1, ..., m\}$. Also, ${\bf k}_\ell$ denotes the field of
characteristic $\ell$ and ${\bf k}$ denotes a field of arbitrary
characteristic. Let
$$2^{[m]*}=\big\{f \big| f:2^{[m]}\To\Z/2\Z=\{0,1\}\big\}$$
consisting of all $\Z/2\Z$-valued functions on the power set
$2^{[m]}$. $2^{[m]*}$ forms an algebra over $\Z/2\Z$  in the usual
way, and it has a natural basis $\{\delta_a| a\in 2^{[m]}\}$ where
$\delta_a$ is defined as follows: $\delta_a(b)=1\Longleftrightarrow
b=a$.   Given a $f\in 2^{[m]*}$, the inverse image of $f$ at 1 is
called the {\em support} of $f$, denoted by $\text{supp}(f)$. $f$ is
said to be {\em nice} if $\text{supp}(f)$ is an abstract simplicial
complex. Thus, we can identify all nice functions in $2^{[m]*}$ with
all abstract simplicial subcomplexes in $2^{[m]}$. On $2^{[m]*}$, we
then  define a $\Z/2\Z$-valued M\"obius transform $\Mob:
2^{[m]*}\longrightarrow 2^{[m]*}$ by the following way: for any
$f\in 2^{[m]*}$ and $a\in 2^{[m]}$, $\Mob(f)(a)=\sum_{b\subseteq
a}f(b)$.

\vskip .2cm

 Now let $f\in 2^{[m]*}$ be nice such that
$K_f=\text{supp}(f)$ is an abstract simplicial complex on $[m]$, and
let ${\bf k}(K_f)$ be the Stanley--Reisner face ring of $K_f$. The
following result indicates an essential relationship between
$\Mob(f)$ and the Betti numbers of  ${\bf k}(K_f)$.

\begin{thm}[Algebra--combinatorics formula]\label{alg-comb}
Suppose that $f\in 2^{[m]*}$ is nice such that
$K_f=\mathrm{supp}(f)$ is an abstract simplicial complex on $[m]$.
Then
$$\Mob(f)=\sum_{i=0}^h\sum_{a\in 2^{[m]}} \beta_{i, a}^{{\bf k}(K_f)}\delta_a$$
where  $h$ denotes the length of the minimal free resolution of
${\bf k}(K_f)$, and $\beta_{i, a}^{{\bf k}(K_f)}$'s denote the Betti
numbers of ${\bf k}(K_f)$ $($see Definition~\ref{betti n}$)$.
\end{thm}

The formula of Theorem~\ref{alg-comb} leads to the following
inequality
$$|\text{supp}(\Mob(f))|\leq \sum_{i=0}^h\sum_{a\in 2^{[m]}} \beta_{i, a}^{{\bf k}(K_f)}.$$
See Corollary~\ref{bt}.  Then we use an approach  of compressing
$\mathrm{supp}(f)$ to further analyze the lower bound of
$|\text{supp}(\Mob(f))|$, and the result is stated as follows.

\begin{thm}\label{lower-bound}
Let $f\in 2^{[m]*}$ be nice such that $K_f=\mathrm{supp}(f)$ is an
abstract simplicial complex on $[m]$. Then there exists some $a\in
\mathrm{supp}(f)$ such that
$$|\mathrm{supp}(\Mob(f))|\geq 2^{m-|a|}.$$
\end{thm}

\begin{rem}
Since $a\in \text{supp}(f)$, $|a|\leq \dim K_f+1$, so
$|\mathrm{supp}(\Mob(f))|\geq 2^{m-|a|}\geq 2^{m-\dim K_f-1}$.
\end{rem}

As a result, we can consider the Halperin--Carlsson conjecture in the
category of (real) moment-angle complexes. Let $\mathcal{Z}_K$
(resp. ${\Bbb R}\mathcal{Z}_K$)  be the moment-angle complex (resp.
real moment-angle complex) on $K$ where $K$ is an abstract
simplicial complex on vertex set $[m]$. Then we have that $ \sum_i
\dim_{{\bf k}} H^i(\mathcal{Z}_K; {\bf k})=\sum_i \dim_{{\bf k}} H^i({\Bbb R}\mathcal{Z}_K; {\bf
k})$ for any ${\bf k}$ (see Theorem~\ref{module-str}). $\mathcal{Z}_K$ (resp.
${\Bbb R}\mathcal{Z}_K$) naturally admits a $T^m$-action $\Phi$
(resp. $(\Z_2)^m$-action $\Phi_{\Bbb R}$).
\begin{thm}\label{main}
Let $H$ $($resp. $H_{\Bbb R})$ be a rank $r$ subtorus of $T^m$
$($resp. $(\Z_2)^m)$. If $H$ $($resp. $H_{\Bbb R})$ can act freely
on $\mathcal{Z}_K$ $($resp. ${\Bbb R}\mathcal{Z}_K)$, then $$ \sum_i
\dim_{{\bf k}} H^i(\mathcal{Z}_K; {\bf k})=\sum_i \dim_{{\bf k}} H^i({\Bbb R}\mathcal{Z}_K; {\bf
k})\geq 2^r.$$
\end{thm}

\begin{rem}
In Theorem~\ref{main}, the action of $H$ (resp. $H_{\Bbb R}$) on
$\mathcal{Z}_K$ (resp. ${\Bbb R}\mathcal{Z}_K$) is naturally
regarded as the restriction of the $T^m$-action $\Phi$ to $H$ (resp.
the $(\Z_2)^m$-action $\Phi_{\Bbb R}$ to $H_{\Bbb R}$).
\end{rem}

\begin{cor}\label{H-C conj}
The Halperin--Carlsson conjecture holds for $\mathcal{Z}_K$ $($resp.
${\Bbb R}\mathcal{Z}_K)$ under the restriction of the $T^m$-action
$\Phi$ $($resp. the $({\Bbb Z}_2)^m$-action $\Phi_{\Bbb R})$.
\footnote[2]{\ T. E. Panov informs of us that
 using a different method, Yury Ustinovsky has also
recently
proved the Halperin's toral rank conjecture for the moment-angle complexes with the restriction of natural tori actions, see arXiv:0909.1053.}
\end{cor}

\begin{rem}
Following~\cite{p}, the Halperin--Carlsson conjecture is stated as
follows:
\begin{enumerate}
\item[$\bullet$]
 Let $X$ be  a finite-dimensional
paracompact Hausdorff space. If  $X$ admits a free action of a torus
$T^r$ (resp. a $p$-torus $({\Bbb Z}_p)^r, p$ prime) of rank $r$,
then
\begin{equation}\label{conjecture}
\sum_i \dim_{{\bf k}_\ell} H^i(X; {\bf
k}_\ell)\geq 2^r
\end{equation} where $\ell$ is $0$ (resp. $p$).
\end{enumerate}
Historically, the above conjecture in the $p$-torus case originates
from the  work of P. A. Smith (\cite{s}). For the case of a
$p$-torus $(\Z_p)^r$ freely acting on a finite CW-complex homotopic
to $(S^n)^k$ suggested by P. E. Conner (\cite{c}), the problem has
made an essential progress (see \cite{ab}, \cite{ca1}--\cite{ca2}
and \cite{y}). In the general case, the inequality
(\ref{conjecture}) was conjectured by S. Halperin in \cite{h} for
the torus case, and by G. Carlsson in \cite{ca3} for the $p$-torus
case. So far, the conjecture holds if $r\leq 3$ in the torus and
2-torus cases and if $r\leq 2$ in the odd $p$-torus case (see
\cite{p}). Also, many authors have given contributions to the
conjecture in many different aspects. For more details, see, e.g.,
\cite{ad}--\cite{ap}, \cite{bp1} and~\cite{pa}.
\end{rem}

The paper is organized as follows. In Section~\ref{s2} we  study the
basic structure of the algebra $2^{[m]*}$ and the  basic properties
of the $\Z/2\Z$-valued M\"obius transform, and review the notions
 of Stanley--Reisner face rings and their Tor-algebras.
Section~\ref{s3} is the main part of this paper. We give the proof
of the algebra--combinatorics formula and estimate the lower bound
of $|\mathrm{supp}(\Mob(f))|$ therein. In Section~\ref{s4} we
introduce the definitions of $\mathcal{Z}_K$ and ${\Bbb
R}\mathcal{Z}_K$, and review the theorem of V. M. Buchstaber and T.
E. Panov on the cohomology of $\mathcal{Z}_K$. In particular, we
also calculate the  cohomology (as a graded ${\bf k}$-module)  of
the generalized moment-angle complex
$\mathcal{Z}_K^{(\underline{{\Bbb D}}, \underline{{\Bbb S}})}$, see
Subsection~\ref{cohomology} for the definition of
$\mathcal{Z}_K^{(\underline{{\Bbb D}}, \underline{{\Bbb S}})}$.
Finally we finish the proof of Theorem~\ref{main}  in
Section~\ref{s5}.

\section{M\"obius transform and Stanley--Reisner face
ring}\label{s2}

\subsection{An algebra over $\Z/2\Z$}

 Let
$2^{[m]}$ denote the power set of $[m]$, which is the set of all
subsets (including the empty set) of $[m]$. Then $2^{[m]}$ forms a
poset with respect to the inclusion $\subseteq$, and it is also a
boolean algebra under the set operations of union, intersection and
complement relative to $[m]$. Let
$$2^{[m]*}=\big\{f \big| f:2^{[m]}\To\Z/2\Z=\{0,1\}\big\}.$$
 Then $2^{[m]*}$ forms
{\em an algebra over $\Z/2\Z$}, where the addition is defined by
$(f+g)(a)=f(a)+g(a)$ and the multiplication is defined by $(f\cdot
g)(a)=f(a)g(a)$ for $ a\in 2^{[m]}$. Given a function $f\in
2^{[m]*}$, define
$$\mathrm{supp}(f):=f^{-1}(1)$$
which is called the {\em support} of $f$.

\begin{defn}
 For each $a\in 2^{[m]}$, the function $\delta_a\in 2^{[m]*}$ defined by
$$\delta_a(b)=\begin{cases}
1 & \text{if }  b=a\\
0 & \text{otherwise} \end{cases}$$ is called {\em the $a$-function}.
For each $i\in [m]$, the function $x_i\in 2^{[m]*}$ defined by
$$x_i(a)=1\Leftrightarrow i\in a $$  for  $\forall \ a\in 2^{[m]}$ is called  {\em the $i$-th coordinate
function}.
\end{defn}

\begin{lem}\label{basis}
$\{\delta_a| a\in 2^{[m]}\}$ forms a basis for $2^{[m]*}$.
\end{lem}
\begin{proof}
This is because any $f\in 2^{[m]*}$ can be expressed as
$$f=\sum_{a\in
2^{[m]}}f(a)\delta_a=\sum_{a\in\mathrm{supp}(f)}\delta_a.$$
\end{proof}

By $\underline{1}$ one denotes the constant function such that
$\underline{1}(a)=1$ for all $a$ in $2^{[m]}$. Obviously,
$\underline{1}=\sum_{a\in 2^{[m]}}\delta_a$.  For each $a\in
2^{[m]}$, set
 $$\mu_a:=\begin{cases}
 \prod_{i\in
a}x_i & \text{if $a$ is nonempty}\\
\underline{1} &\text{if $a$ is empty .} \end{cases}$$  Then it is
easy to see that
\begin{lem}\label{mu-function}
Let $a, b\in 2^{[m]}$. Then
 $\mu_a(b)=1\Leftrightarrow a\subseteq b$.
\end{lem}

\begin{defn}
$f\in 2^{[m]*}$ is said to be  \emph{nice} if $\mathrm{supp}(f)$ is
an abstract simplicial complex on  vertex set $\bigcup_{a\in
\text{supp}(f)}a\subseteq[m]$. Note that an abstract simplicial
complex $K$ on a subset of $[m]$ is a collection of subsets in $[m]$
with the property that for each $a\in K$, all subsets (including the
empty set) of $a$ belong to $K$. Each $a\in K$ is called a {\em
simplex} and has dimension $|a|-1$. The dimension of $K$ is defined
as $\max_{a\in K}\{\dim a\}$.
\end{defn}

It is easy to see that $f$ is nice if and only if for each $a\in
\text{supp}(f)$, any subset $b\subseteq a$ has  the property
$f(b)=1$.

\vskip .2cm

Let $\mathcal{F}_{[m]}=\{f\in 2^{[m]*}| f \text{ is nice} \}$, and
$\mathcal{K}_{[m]}$ the set of all abstract simplicial complexes on
vertex set $A$ where  $A$ runs over all possible subsets in $[m]$.
\begin{prop}\label{1-1 corr}
All functions of $\mathcal{F}_{[m]}$ bijectively correspond to all
abstract simplicial complexes of $\mathcal{K}_{[m]}$.
\end{prop}
\begin{proof}
Clearly,   $f\mapsto\mathrm{supp}(f)$ gives a bijection
$\mathcal{F}_{[m]}\To\mathcal{K}_{[m]}$, whose inverse is
$K\mapsto\sum_{a\in K}\delta_a$.
\end{proof}

\subsection{M\"obius transform} Based upon Proposition~\ref{1-1 corr}, we shall carry out our work from the viewpoint of functional analysis.

\begin{defn}
 The map $\Mob: 2^{[m]*}\To 2^{[m]*}$ given by the
formula
$$\Mob(f)(a)=\sum_{b\subseteq a}f(b)$$  for all  $f\in 2^{[m]*}$  and $a\in
2^{[m]}$ is called the \emph{$\Z/2\Z$-valued M\"obius transform}.
\end{defn}

\begin{lem}\label{mobius}
$\mathrm{\Mob}$ is a linear transform such that
$\mathrm{\Mob^2=\mathrm{id}}$. In particular,
\begin{equation}\label{eq:Mob(del)}\mathrm{\Mob}(\delta_a)=\mu_a\end{equation} for any ${a\in
2^{[m]}}$. Consequently, $\mathrm{\Mob}(\mu_a)=\delta_a$.
\end{lem}

\begin{proof}
The linearity of $\mathrm{\Mob}$ is obvious. To check that
$\Mob^2=\mathrm{id}$, take $f\in 2^{[m]*}$, one has that for any
$a\in 2^{[m]}$
\begin{equation}\label{equation}
\mathcal{M}^2(f)(a)=\sum_{b\subseteq a}\sum_{c\subseteq b}f(c)
                =\sum_{c\subseteq a}\sum_{b\in[c,a]}f(c)
                =f(a)+\sum_{c\subsetneq a}\sum_{b\in[c,a]}f(c)
\end{equation}
For every term in the latter sum of (\ref{equation}), from
$c\subsetneq a$ we see that $[c,a]$ is a boolean subalgebra of
$2^{[m]}$ which has $2^k$ elements for some $k>0$. So the sum
$\sum_{b\in[c,a]}f(c)=0$ in $\Z/2\Z$. Therefore
$\mathcal{M}^2(f)(a)=f(a)$ for any $a\in 2^{[m]}$, so
$\mathcal{M}^2(f)=f$ as desired. The equation (\ref{eq:Mob(del)}) is
a direct calculation by Lemma~\ref{mu-function}.
\end{proof}

As a consequence of Lemmas~\ref{basis} and~\ref{mobius}, one has

\begin{cor}\label{basis'}
 ${\{\mu_a | a\in 2^{[m]}\}}$ is also a basis of $2^{[m]*}$.
\end{cor}

\begin{rem}
By definition of $\Mob$, if $f(\emptyset)=1$ then
$\Mob(f)(\emptyset)=1$.
\end{rem}

In the next two subsections we shall review the Stanley--Reisner face
rings and Tor-algebras. Our main reference is the book by E. Miller
and  B. Sturmfels (\cite{ms}).
\subsection{Stanley--Reisner face
ring}\label{st}

Now let $f\in \mathcal{F}_{[m]}$ be a nice function such that
$K_f=\text{supp}(f)\in\mathcal{K}_{[m]}$ is an abstract simplicial
complex on $[m]$ (so $\bigcup_{a\in K_f}a=[m]$).

\vskip .2cm

 Following the notions of \cite{ms}, let ${\bf k}[{\bf v}]={\bf
k}[v_1, ..., v_m]$ be the polynomial algebra over ${\bf k}$ on $m$
indeterminates ${\bf v}=v_1, ..., v_m$. Each monomial in ${\bf
k}[{\bf v}]$ has the form of  ${\bf v}^{\bf a}=v_1^{a_1}\cdots
v_m^{a_m}$ for a vector ${\bf a}=(a_1, ..., a_m)\in {\Bbb N}^m$ of
nonnegative integers. Thus, ${\bf k}[{\bf v}]$ is ${\Bbb
N}^m$-graded, i.e., ${\bf k}[{\bf v}]$ is a direct sum
$\bigoplus_{{\bf a}\in {\Bbb N}^m}{\bf k}[{\bf v}]_{\bf a}$ with
${\bf k}[{\bf v}]_{\bf a}\cdot{\bf k}[{\bf v}]_{\bf b}={\bf k}[{\bf
v}]_{{\bf a}+{\bf b}}$ where ${\bf k}[{\bf v}]_{\bf a}={\bf k}\{{\bf
v}^{\bf a}\}$ is the  vector space over ${\bf k}$, spanned by ${\bf
v}^{\bf a}$. Generally, a ${\bf k}[{\bf v}]$-module $M$ is {\em
${\Bbb N}^m$-graded} if $M = \bigoplus_{{\bf b}\in {\Bbb N}^m}M_{\bf
b}$ and ${\bf v}^{\bf a}\cdot M_{\bf b}\subseteq M_{{\bf a}+{\bf
b}}$. Given a vector ${\bf a}\in {\Bbb N}^m$, by ${\bf k}[{\bf
v}](-{\bf a})$ one denotes the free ${\bf k}[{\bf v}]$-module
generated in degree ${\bf a}$. So ${\bf k}[{\bf v}](-{\bf a})$ is
isomorphic to the ideal $\langle{\bf v}^{\bf a}\rangle$ as ${\Bbb
N}^m$-graded modules. Furthermore, a free ${\Bbb N}^m$-graded module
of rank $r$ is isomorphic to the direct sum ${\bf k}[{\bf v}](-{\bf
a}_1)\bigoplus\cdots\bigoplus{\bf k}[{\bf v}](-{\bf a}_r)$ for some
vectors ${\bf a}_1, ..., {\bf a}_r\in {\Bbb N}^m$.
 \vskip .2cm

A monomial ${\bf v}^{\bf a}$ in ${\bf k}[{\bf v}]$ is said to be
{\em squarefree} if every coordinate of ${\bf a}$ is 0 or 1, i.e.,
${\bf a}\in \{0, 1\}^m$ called a {\em squarefree vector}. Clearly,
all elements  in $2^{[m]}$ bijectively correspond to all vectors  in
$\{0, 1\}^m$ by mapping $\xi: a\in 2^{[m]}\longmapsto{\bf a}\in
\{0,1\}^m$, where ${\bf a}$ has entry 1 in the $i$-th place when
$i\in a$, and 0 in all other entries. With this understanding, for
$a\in 2^{[m]}$, one may write ${\bf v}^a=\prod_{i\in a}v_i$.
 Then the
{\em Stanley--Reisner ideal} of $K_f$ is defined as $I_{K_f}=\langle
{\bf v}^\tau| \tau\not\in K_f\rangle$. Furthermore, the quotient
ring
$${\bf k}(K_f)={\bf k}[{\bf v}]/I_{K_f}$$
is called the {\em Stanley--Reisner face ring}.

\begin{example}
If $K_f=2^{[m]}$ then ${\bf k}(K_f)={\bf k}[{\bf v}]$, and if
$K_f=2^{[m]}\setminus\{[m]\}$ then ${\bf k}(K_f)={\bf k}[{\bf
v}]/\langle{\bf v}^{[m]}\rangle$.
\end{example}

 ${\bf k}(K_f)$ is a finitely generated graded ${\bf k}[{\bf
v}]$-module. Hilbert's syzygy theorem tells us that there exists a
 free resolution of ${\bf k}(K_f)$ of length at most $m$.
One knows from \cite[Section 1.4]{ms} that ${\bf k}(K_f)$ is  ${\Bbb
N}^m$-graded  and it has an ${\Bbb N}^m$-graded minimal free
resolution as follows
\begin{equation}\label{fr}
0\longleftarrow {\bf k}(K_f)\longleftarrow F_0
\stackrel{\phi_1}{\longleftarrow}
F_1\longleftarrow\cdots\longleftarrow
F_{h-1}\stackrel{\phi_h}{\longleftarrow} F_h\longleftarrow 0
\end{equation}
  where each homomorphism $\phi_i$ is ${\Bbb N}^m$-graded degree-preserving.
Since each $F_i$ is a free ${\Bbb N}^m$-graded ${\bf k}[{\bf
v}]$-module,  one may write $F_i=\bigoplus_{{\bf a}\in {\Bbb
N}^m}{\bf k}[{\bf v}](-{\bf a})^{\beta_{i, {\bf a}}^{{\bf k}(K_f)}}$
where $\beta_{i, {\bf a}}^{{\bf k}(K_f)}\in {\Bbb N}$ (see also
\cite[Section 1.5]{ms}). By \cite[Corollary 1.40]{ms}, if ${\bf
a}\in {\Bbb N}^m$ is not  squarefree, then $\beta_{i, {\bf a}}^{{\bf
k}(K_f)}=0$ for all $i$. Thus, we only need to consider those
$\beta_{i, {\bf a}}^{{\bf k}(K_f)}$ with ${\bf a}\in \{0,1\}^m$.
Throughout the following  we shall write $\beta_{i,a}^{{\bf
k}(K_f)}:=\beta_{i, {\bf a}}^{{\bf k}(K_f)}$ where $a\in 2^{[m]}$
with $\xi(a)={\bf a}$.
\begin{defn}[{cf.~\cite[Definition 1.29]{ms}}]\label{betti n}
The number $\beta_{i, {a}}^{{\bf k}(K_f)}$ is called the {\em $(i,
a)$-th Betti number} of ${\bf k}(K_f)$.
\end{defn}

\subsection{Tor-algebra of ${\bf k}(K_f)$}\label{tor}
Applying the functor $\otimes_{{\bf k}[{\bf v}]}{\bf k}$ to the
sequence (\ref{fr}),  one may obtain the following chain complex of
${\Bbb N}^m$-graded ${\bf k}[{\bf v}]$-modules:
$$
0\longleftarrow  F_0\otimes_{{\bf k}[{\bf v}]}{\bf k}
\stackrel{\phi'_1}{\longleftarrow} F_1\otimes_{{\bf k}[{\bf v}]}{\bf
k}\longleftarrow\cdots\stackrel{\phi'_h}{\longleftarrow}
F_h\otimes_{{\bf k}[{\bf v}]}{\bf k}\longleftarrow 0.
$$
Since the free resolution (\ref{fr}) is minimal,  the differentials
$\phi'_i$'s become zero homomorphisms.  Then the $i$-th homology
module of the above chain complex is ${{\ker\phi'_i}\over
{\text{Im}\phi'_{i+1} }}=F_i\otimes_{{\bf k}[{\bf v}]}{\bf k}$,
denoted by $\text{Tor}^{{\bf k}[{\bf v}]}_i({\bf k}(K_f), {\bf k})$.
Namely,  $\text{Tor}^{{\bf k}[{\bf v}]}_i({\bf k}(K_f), {\bf
k})=F_i\otimes_{{\bf k}[{\bf v}]}{\bf k}$ so
$$\dim_{\bf k}\text{Tor}^{{\bf k}[{\bf v}]}_i({\bf k}(K_f), {\bf
k})=\text{rank} F_i=\sum_{a\in 2^{[m]}}\beta_{i,a}^{{\bf k}(K_f)}.$$
This also implies that for ${\bf a}\in {\Bbb N}^m$ with ${\bf
a}\not\in \{0,1\}^m$, $\text{Tor}^{{\bf k}[{\bf v}]}_i({\bf k}(K_f),
{\bf k})_{\bf a}=0$, and so  $\text{Tor}^{{\bf k}[{\bf v}]}_i({\bf
k}(K_f), {\bf k})$ can  be decomposed into a direct sum
$$\bigoplus_{a\in 2^{[m]}}\text{Tor}^{{\bf k}[{\bf v}]}_{i}({\bf
k}(K_f), {\bf k})_a$$ with $\dim_{\bf k}\text{Tor}^{{\bf k}[{\bf
v}]}_i({\bf k}(K_f), {\bf k})_a=\beta_{i,a}^{{\bf k}(K_f)}$ (see
also \cite[Lemma 1.32]{ms}). Furthermore, one has that
$$\text{Tor}^{{\bf k}[{\bf v}]}({\bf
k}(K_f), {\bf k})=\bigoplus_{i=0}^h\text{Tor}_i^{{\bf k}[{\bf
v}]}({\bf k}(K_f), {\bf k})=\bigoplus_{i\in [0, h]\cap {\Bbb N},
a\in 2^{[m]}}\text{Tor}^{{\bf k}[{\bf v}]}_{i}({\bf k}(K_f), {\bf
k})_a$$ which is a bigraded ${\bf k}[{\bf v}]$-module. Combining
with the above arguments, this gives
\begin{prop}\label{tor-betti}
$\sum\limits_{i=0}^h\dim_{\bf k}\text{\rm Tor}_i^{{\bf k}[{\bf
v}]}({\bf k}(K_f), {\bf k})=\sum\limits_{i=0}^h\sum\limits_{a\in
2^{[m]}}\beta_{i,a}^{{\bf k}(K_f)}.$
\end{prop}

\section{M\"obius transform of abstract simplicial complexes and Betti numbers of face
rings}\label{s3}

\subsection{An algebra--combinatorics formula}
Following Subsections~\ref{st}--\ref{tor}, now let us  investigate
the essential relationship between the M\"obius transform $\Mob(f)$
of $f$ and the Betti numbers of the face ring ${\bf k}(K_f)$ of
$K_f$.

\begin{thm}[Algebra--combinatorics formula]\label{p:alg-comb}
$$\Mob(f)=\sum_{i=0}^h\sum_{a\in 2^{[m]}} \beta_{i, a}^{{\bf k}(K_f)}\delta_a.$$
\end{thm}
\begin{proof}
For any $b\in 2^{[m]}$, the exact sequence (\ref{fr}) in degree $b$
reads into
$$0\longleftarrow {\bf k}^{D_{b}}\longleftarrow {\bf k}^{d_{b,0}}\longleftarrow {\bf k}^{d_{b,1}}\longleftarrow
{\bf k}^{d_{b,2}}\longleftarrow\cdots\longleftarrow {\bf k}^{d_{b,
h}}\longleftarrow 0$$ where $D_{b}=\dim_{\bf k}{\bf k}(K_f)_b$ and
$d_{b,i}=\dim_{\bf k}(F_i)_b$.  Since the above sequence is also
exact, we have that $D_{b}=\sum_{i=0}^h {(-1)}^{i} d_{i,b}$. An easy
observation shows that $f(b)=\dim_{\bf k}{\bf k}(K_f)_b=D_b$, and
$d_{b, i}=\sum_{a\subseteq b} \beta_{i,a}^{{\bf k}(K_f)}$ (this is
induced from $F_i=\bigoplus_{{\bf a}\in {\Bbb N}^m}{\bf k}[{\bf
v}](-{\bf a})^{\beta_{i, {\bf a}}^{{\bf k}(K_f)}}$).

\vskip .2cm

Now let us work in integers modulo 2. We then have that
 $D_{b}=\sum_{i=0}^h
d_{i,b}$, and further
$$
f(b)= \sum_{i=0}^h \sum_{a\subseteq b} \beta_{i,a}^{{\bf k}(K_f)}
      =\sum_{i=0}^h \sum_{a\in 2^{[m]}}  \beta_{i,a}^{{\bf k}(K_f)} \mu_a(b).
$$
So \begin{equation*}\label{formu} f=\sum_{i=0}^h\sum_{a\in
2^{[m]}}\beta_{i,a}^{{\bf k}(K_f)}\mu_a. \end{equation*}
Applying $\Mob$ to the above equality  and noting that
$\Mob(\mu_a)=\delta_a$, we arrive at the required formula.
\end{proof}

\begin{cor}\label{bt}
Let $f\in 2^{[m]*}$ be a nice function such that
$K_f=\mathrm{supp}(f)\in\mathcal{K}_{[m]}$ is an abstract simplicial
complex on $[m]$. Then
$$
|\mathrm{supp}(\Mob(f))|\leq \sum_{i=0}^h\sum_{a\in 2^{[m]}} \beta_{i, a}^{{\bf k}(K_f)}.$$
\end{cor}

\begin{proof}
From the formula of Theorem~\ref{p:alg-comb}, one has that
$$\Mob(f)=\sum_{i=0}^h\sum_{a\in 2^{[m]}} \beta_{i, a}^{{\bf k}(K_f)}\delta_a=
\sum_{a\in 2^{[m]}}\Big(\sum_{i=0}^h \beta_{i, a}^{{\bf
k}(K_f)}\Big)\delta_a$$ so
 for any
$a\in \mathrm{supp}(\Mob(f))$,  $\sum_{i=0}^h\beta_{i, a}^{{\bf
k}(K_f)}$ must be odd and nonnegative, and then
$\sum_{i=0}^h\beta_{i,a}^{{\bf k}(K_f)}\geq 1$.
 Therefore
$$\sum_{i=0}^h\sum_{a\in 2^{[m]}} \beta_{i, a}^{{\bf k}(K_f)}\geq \sum_{a\in \mathrm{supp}(\Mob(f))}
\sum_{i=0}^h\beta_{i,a}^{{\bf k}(K_f)}\geq \sum_{a\in
\mathrm{supp}(\Mob(f))} 1=|\mathrm{supp}(\Mob(f))|$$ as desired.
\end{proof}

\subsection{The estimation of the lower bound of
$|\text{supp}(\Mob(f))|$} We shall upbuild a  method of compressing
$\mathrm{supp}(f)$ to get the desired lower bound of
$|\text{supp}(\Mob(f))|$.

\begin{defn}
Fix $k\in[m]$.  $f\in\mathcal{F}_{[m]}$ is \emph{$k$-th extendable}
if
\begin{enumerate}
\item[(k-1)] $f(\{k\})=1$;
\item[(k-2)] $\Mob(f)\cdot x_k\neq 0$ in $2^{[m]*}$.
\end{enumerate}
 The linear transformation $E_k:
2^{[m]*}\To\ 2^{[m]*}$ determined by
$\mu_a\mapsto\mu_{a\setminus\{k\}}$ is called the \emph{$k$-th
compression-operator}. $f\in\mathcal{F}_{[m]}$ is said to be {\em
extendable} if there is some $k\in [m]$ such that $f$ is
$k$-th extendable; otherwise, $f$ is said to be {\em non-extendable}.
\end{defn}
Introducing the map $\epsilon_k: 2^{[m]} \To 2^{[m]}$ defined by
$a\mapsto a\cup\{k\}$, we derive the following formula for $E_k$.

\begin{lem}\label{E_k}
 For any $f\in 2^{[m]*}$ we have
\begin{equation*}E_k(f)=f\circ\epsilon_k.\end{equation*}
\end{lem}
\begin{proof}
It suffices to check that the formula
$E_k(\mu_a)=\mu_a\circ\epsilon_k$ holds for  each $a\in 2^{[m]}$.
Indeed, take $b\in 2^{[m]}$, we have that
$$E_k(\mu_a)(b)=1\Leftrightarrow a\setminus\{k\}\subseteq
b\Leftrightarrow a\subseteq b\cup\{k\}\Leftrightarrow
\mu_a(\epsilon_k(b))=1.$$ Therefore,
$E_k(\mu_a)=\mu_a\circ\epsilon_k$ as desired.
\end{proof}

\begin{prop} \label{p:E_k}
Fix $k\in[m]$. If $f\in\mathcal{F}_{[m]}$ satisfies $f(\{k\})=1$,
then $E_k(f)\in\mathcal{F}_{[m]}$ and
$\mathrm{supp}(E_k(f))\subseteq \mathrm{supp}(f)$.
\end{prop}
\begin{proof}
For any pair $a\subseteq b$ in $2^{[m]}$, we have that
$\epsilon_k(a)\subseteq\epsilon_k(b)$. So if $E_k(f)(b)=
f(\epsilon_k(b))=1$, then $f(\epsilon_k(a))=1$ since
$f\in\mathcal{F}_{[m]}$, and so $E_k(f)(a)=1$. Also, $f(\{k\})=1$
implies that $E_k(f)(\emptyset)=f(\emptyset\cup\{k\})=1$.  Thus,
$E_k(f)$ is nice.

\vskip .2cm

For any $a\in 2^{[m]}$, if $E_k(f)(a)=1$ then by Lemma~\ref{E_k}
$f(\epsilon_k(a))=1$, so $f(a)=1$ since $a\subseteq \epsilon_k(a)$
and $f\in\mathcal{F}_{[m]}$. Hence,  $\mathrm{supp}(E_k(f))\subseteq
\mathrm{supp}(f)$ as desired.
\end{proof}

Now let us look at the composition transformation $\Mob\circ
E_k\circ \Mob=:\hat{E}_k$. For any $a\in 2^{[m]}$, one has
\begin{equation}\label{inv}
\hat{E}_k(\delta_a)=\Mob\circ E_k\circ \Mob(\delta_a)=\Mob\circ E_k(\mu_a)
=\Mob(\mu_{a\setminus\{k\}})=\delta_{a\setminus\{k\}}.
\end{equation}
Note also that since $\Mob^2=\mathrm{id}$, $\Mob\circ
E_k=\hat{E}_k\circ\Mob$.

\begin{lem}\label{hat{E}_k}
For any $g\in 2^{[m]*}$ and $k\in [m]$, $\hat{E}_k(g)x_k=0$.
\end{lem}
\begin{proof}
Write $g=\sum_{a\in \text{supp}(g)}\delta_a$. Since $\hat{E}_k$ is
linear and $\hat{E}_k(\delta_a)=\delta_{a\setminus\{k\}}$ for any
$a\in 2^{[m]}$, it follows that $\hat{E}_k(g)=\sum_{a\in
\text{supp}(g)}\delta_{a\setminus \{k\}}$. Obviously, for any $a\in
2^{[m]}$, $\delta_{a\setminus \{k\}}x_k=0$. Thus,
$\hat{E}_k(g)x_k=0$ as desired.
\end{proof}

\begin{cor}\label{c:Neq}
Let $k\in [m]$. If $f\in 2^{[m]*}$ satisfies
$\Mob(f)x_k\not=0$, then $f\neq E_k(f)$.
\end{cor}

\begin{proof}
Suppose that $f=E_k(f)$. Applying $\Mob$ to both sides, we get
$\Mob(f)=\Mob (E_k(f))=\hat{E}_k(\Mob (f))$.  Write $g=\Mob(f)$. Then
$g=\hat{E}_k(g)$. Multiplied by $x_k$ on the two sides of
$g=\hat{E}_k(g)$, we have that $gx_k=\hat{E}_k(g)x_k$. Since
$gx_k=\Mob(f)x_k\not=0$,  we have $\hat{E}_k(g)x_k\neq0$, a
contradiction  by Lemma~\ref{hat{E}_k}.
\end{proof}

\begin{prop} \label{p:mob(f)}
Let $f\in 2^{[m]*}$. Then for each $k\in [m]$,
\begin{equation*}|\mathrm{supp}(\hat{E}_k(f))|\leq|\mathrm{supp}(f)|.\end{equation*}
\end{prop}
\begin{proof}
Let $A=\big\{a\in 2^{[m]}\big| k\notin a, a\in\mathrm{supp}(f)\big\}$ and
$B=\big\{a\in 2^{[m]}\big| k\notin a,\epsilon_k(a)\in\mathrm{supp}(f)\big\}$.
Then we have
\begin{equation*}
\begin{split}
f=\sum_{a\in\mathrm{supp}(f)}\delta_a
 =\sum_{a\in\mathrm{supp}(f) \atop k\notin a} \delta_a+\sum_{a\in\mathrm{supp}(f) \atop k\in a} \delta_a
 =\sum_{a\in A} \delta_a+ \sum_{a\in B} \delta_{\epsilon_k(a)} \\
\end{split}
\end{equation*}
and by (\ref{inv})
\begin{equation*}
\begin{split}
 \hat{E}_k(f)=&\sum_{a\in A}\delta_{a\setminus\{k\}}+\sum_{a\in B}
 \delta_{\epsilon_k(a)\setminus\{k\}}
 =\sum_{a\in A}\delta_{a}+\sum_{a\in B}
 \delta_{a}=\sum_{a\in A\bigtriangleup B} \delta_{a}
\end{split}
\end{equation*}
where $ A\bigtriangleup B=(A\setminus B)\cup (B\setminus A)$.  Now
$$|\mathrm{supp}(\hat{E}_k(f))|=|A\bigtriangleup
B|\leq|A|+|B|=|\mathrm{supp}(f)|$$ as desired.
\end{proof}

\begin{rem}\label{sequence}
Observe that for any $f\in\mathcal{F}_{[m]}$, whenever $f$ is
$k$-th extendable for some $k\in [m]$, by Proposition~\ref{p:E_k} and
Corollary~\ref{c:Neq} we obtain that $E_k(f)\in\mathcal{F}_{[m]}$
and $\mathrm{supp}(E_k(f))\subsetneq\mathrm{supp}(f)$. In addition,
since $(\mathcal{M}\circ E_k)(f)=(\hat{E}_k\circ \mathcal{M})(f)$,
by Proposition~\ref{p:mob(f)} one has that
$|\mathrm{supp}(\mathcal{M}(E_k(f)))|\leq|\mathrm{supp}(\mathcal{M}(f))|$. We replace $f$ with
$E_k(f)$ and repeat the above process whenever possible, so as to
get a sequence of functions in $\mathcal{F}_{[m]}$ with strictly
decreasing support. This process must end after a finite number of
steps, giving finally a $f_0\in\mathcal{F}_{[m]}$ that is \emph{
non-extendable} with $\mathrm{supp}(f_0)\subseteq\mathrm{supp}(f)$
and $|\mathrm{supp}(\mathcal{M}(f_0))|\leq|\mathrm{supp}(\mathcal{M}(f))|$.
It remains to characterize such a non-extendable
$f_0\in\mathcal{F}_{[m]}$.
\end{rem}

\begin{prop}\label{p:non-ex} Let  $f\in\mathcal{F}_{[m]}$. Then $f$ is non-extendable if and only if there is some
$a_0\in 2^{[m]}$  such that $\mathrm{supp}(f)=2^{a_0}$ $($i.e.,
$f=\sum_{b\subseteq a_0}\delta_b)$.
\end{prop}
\begin{proof}
Suppose that $f$ is non-extendable.
 Let $a_0=\big\{k\in [m]\big|
f(\{k\})=1\big\}$. If $a_0=\emptyset$,  obviously we have
$f=\delta_\emptyset$. Assume that $a_0$ is non-empty.  Given an
element $b\in 2^{[m]}$, if $f(b)=1$, since $f\in\mathcal{F}_{[m]}$,
then for any $k\in b, f(\{k\})=1$ so $k\in a_0$ and $b\subseteq
a_0$. Since $f$ is non-extendable,
 $\Mob(f)x_k=0$
for any $k\in a_0$. Then we see from $\Mob(f)=\sum_{b\in
\text{supp}(\Mob(f))}\delta_b$ that for any $b\in
\text{supp}(\Mob(f))$, $b\cap a_0=\emptyset$. Since
$\Mob(f)(\emptyset)=1$, we have that
$\emptyset\in\text{supp}(\Mob(f))$. Furthermore
$$f(a_0)=\Mob^2(f)(a_0)=\Mob\Big(\sum_{b\in
\text{supp}(\Mob(f))}\delta_b\Big)(a_0)=\sum_{b\in
\text{supp}(\Mob(f))}\mu_b(a_0)=\mu_\emptyset(a_0)=1.$$ Since
$f\in\mathcal{F}_{[m]}$, it follows that for any subset $b\subseteq
a_0$, $f(b)=1$. Therefore, for $b\in 2^{[m]}$
$$f(b)=1\Leftrightarrow b\subseteq a_0.$$
This implies that $\text{supp}(f)=2^{a_0}=\{b\in 2^{[m]}\big|
b\subseteq a_0\}$. \vskip .2cm

Conversely, suppose that $f=\sum_{b\subseteq a_0}\delta_b$ for some
$a_0\in [m]$. If $a_0=[m]$, then $f=\underline{1}=\mu_\emptyset$ so
$\Mob(f)=\delta_\emptyset$. Moreover, for any $k\in a_0$,
$\Mob(f)x_k=0$ so $f$ is non-extendable. If $a_0=\emptyset$,
obviously $f$ is non-extendable. Assume that $a_0\not=[m],
\emptyset$. Then an easy argument shows that
$$f=\prod_{i\in [m]\setminus
a_0}(\underline{1}+x_i)=\sum_{b\subseteq [m]\setminus a_0} \mu_b.$$
Applying $\Mob$ to the above equality, it follows that
$\Mob(f)=\sum_{b\subseteq [m]\setminus a_0} \delta_b$. Now for any
$k\in a_0$ and any $b\subseteq [m]\setminus a_0$, we have
$\delta_bx_k=0$ so $\Mob(f)x_k=0$. This means that $f$ is also
non-extendable.
\end{proof}
From the proof of Proposition~\ref{p:non-ex}, we easily see that
\begin{cor}
Let $a\in 2^{[m]}$.  Then $f=\sum_{b\subseteq a}\delta_b$ if and
only if $\Mob(f)=\sum_{b\subseteq [m]\setminus a} \delta_b$ $($i.e.,
$\mathrm{supp}(f)=2^{a}$ if and only if
$\mathrm{supp}(\Mob(f))=2^{[m]\setminus a})$. In this case,
$|\mathrm{supp}(\Mob(f))|=2^{m-|a|}$.
\end{cor}

 We now summarize the above arguments as follows.

\begin{thm}\label{t:extend}
For any $f\in \mathcal{F}_{[m]}$, there exists some $a\in
\mathrm{supp}(f)$ such that
$$|\mathrm{supp}(\Mob(f))|\geq 2^{m-|a|}.$$
\end{thm}

\begin{rem}\label{empty}
The interested readers are invited to see a simple fact that
 $f\in\mathcal{F}_{[m]}$ can be compressed by compression-operators
into a non-extendable $f_0$ with $\mathrm{supp}(f_0)=2^{a_0}$  if and
only if $a_0$ is a maximal element in $\mathrm{supp}(f)$ as a poset.
This result will not be used later in this article.
\end{rem}

\section{Moment-angle complexes and their cohomologies}\label{s4}

Let $K$ be an abstract simplicial complex on vertex set $[m]$. Let
$(X, W)$ be a pair of topological spaces with $W\subset X$.
Following \cite[Construction 6.38]{bp}, for each simplex $\sigma$ in
$K$, set
$$B_\sigma(X, W)=\prod_{i=1}^m A_i$$
such that $$A_i=\begin{cases} X & \text{if } i\in \sigma\\
W & \text{if } i\in [m]\setminus\sigma.
\end{cases}$$ Then one can define the following subspace of the product
space $X^m$:
$$K(X,W)=\bigcup_{\sigma\in K}B_\sigma(X, W)\subset X^m.$$

\subsection{Moment-angle complexes}\label{s4.1} When the pair $(X, W)$ is chosen
as $(D^2, S^1)$,
$$\mathcal{Z}_K:=K(D^2, S^1)\subset (D^2)^m$$
is called the {\em moment-angle complex} on $K$
 where $D^2=\big\{z\in {\Bbb C}\big| |z|\leq 1\big\}$ is the unit disk in ${\Bbb
 C}$, and $S^1=\partial D^2$.
 Since $(D^2)^m\subset {\Bbb C}^m$ is invariant under the
standard action of $T^m$ on ${\Bbb C}^m$ given by $$\big((g_1,...,
g_m), (z_1, ..., z_m)\big)\longmapsto (g_1z_1, ..., g_mz_m),$$
 $(D^2)^m$ admits
a natural $T^m$-action whose orbit space is the unit cube
$I^m\subset {\Bbb R}^m_{\geq 0}$. The action $T^m\curvearrowright
(D^2)^m$ then induces a canonical $T^m$-action $\Phi$ on
$\mathcal{Z}_K$.

\vskip .2cm When the pair $(X, W)$ is chosen as $(D^1, S^0)$,
$${\Bbb R}\mathcal{Z}_K:=K(D^1, S^0)\subset (D^1)^m$$
is called the {\em real moment-angle complex} on $K$
 where $D^1=\big\{x\in {\Bbb R}\big| |x|\leq 1\big\}=[-1, 1]$ is the unit disk in ${\Bbb
 R}$, and $S^0=\partial D^1=\{\pm 1\}$. Similarly, $(D^1)^m\subset {\Bbb R}^m$ is invariant under the
standard action of $({\Bbb Z}_2)^m$ on ${\Bbb R}^m$ given by
$$\big((g_1,..., g_m), (x_1, ..., x_m)\big)\longmapsto (g_1x_1, ...,
g_mx_m).$$ Thus $(D^1)^m$ admits a natural $({\Bbb Z}_2)^m$-action
whose orbit space is also the unit cube $I^m\subset {\Bbb R}^m_{\geq
0}$, where ${\Bbb Z}_2=\{-1, 1\}$ is the group with respect to
multiplication. Furthermore, the action $({\Bbb
Z}_2)^m\curvearrowright (D^1)^m$ also induces a canonical $({\Bbb
Z}_2)^m$-action $\Phi_{\Bbb R}$ on ${\Bbb R}\mathcal{Z}_K$.

\vskip .2cm

Let $P_K$ be the cone on the barycentric subdivision of $K$. Since
the cone on the barycentric subdivision of a $k$-simplex is
combinatorially equivalent to the standard subdivision of a
$(k+1)$-cube, $P_K$ is naturally a cubical complex and it is
decomposed into cubes indexed by the simplices of $K$. Then one
knows from \cite{bp} and \cite{dj} that both $T^m$-action $\Phi$ on
$\mathcal{Z}_K$ and $({\Bbb Z}_2)^m$-action $\Phi_{\Bbb R}$ on
${\Bbb R}\mathcal{Z}_K$ have the same orbit space  $P_K$.

\begin{example}
When $K=2^{[m]}$, $\mathcal{Z}_K=(D^2)^m$ and ${\Bbb
R}\mathcal{Z}_K=(D^1)^m$. When $K=2^{[m]}\setminus\{[m]\}$,
$\mathcal{Z}_K=S^{2m-1}$ and ${\Bbb R}\mathcal{Z}_K=S^{m-1}$.
\end{example}

\begin{rem} In general, $\mathcal{Z}_K$ and ${\Bbb R}\mathcal{Z}_K$
are not manifolds. However,  if $K$ is a simplicial sphere, then
both $\mathcal{Z}_K$ and ${\Bbb R}\mathcal{Z}_K$ are  closed
manifolds (see \cite[Lemma 6.13]{bp}).
\end{rem}

\subsection{Cohomology}\label{cohomology}

V. M. Buchstaber and T. E. Panov in \cite[Theorem 7.6]{bp} have calculated the
cohomology of $\mathcal{Z}_K$ (see also~\cite[Theorem 4.7]{pa}). Their result is stated as follows.
\begin{thm}
[Buchstaber--Panov]\label{bp} As ${\bf k}$-algebras,
$$H^*(\mathcal{Z}_K; {\bf k})\cong \text{\rm Tor}^{{\bf k}[{\bf v}]}({\bf k}(K), {\bf
k})$$
where
${\bf k}(K)={\bf k}[{\bf v}]/I_K={\bf k}[v_1, ...,
v_m]/I_K$ with $\deg v_i=2$.
\end{thm}

Here we shall calculate the cohomologies of a class of generalized moment-angle complexes. For this, we begin with the notion of the generalized moment-angle complex, due to N. Strickland, cf.~\cite{bbcg} and~\cite{ds}.
Given an abstract simplicial complex $K$ on $[m]$, let $(\underline{X}, \underline{W})=\{(X_i, W_i)\}_{i=1}^m$ be $m$
pairs of CW-complexes with $W_i\subset X_i$. Then the {\em generalized moment-angle complex} is defined as follows:
$$K(\underline{X}, \underline{W})=\bigcup_{\sigma\in K}B_\sigma(\underline{X}, \underline{W})\subset
\prod_{i=1}^m X_i$$
where $B_\sigma(\underline{X}, \underline{W})=\prod_{i=1}^m H_i$ and $H_i=\begin{cases}
X_i &\text{ if } i\in \sigma\\
W_i & \text{ if } i\in [m]\setminus\sigma.
\end{cases}$

\vskip .2cm

Now take $(\underline{X}, \underline{W})=(\underline{\Bbb D}, \underline{\Bbb S})=\{({\Bbb D}_i, {\Bbb S}_i)\}_{i=1}^m$ with each CW-complex pair
$({\Bbb D}_i, {\Bbb S}_i)$ subject to the following conditions:
\begin{enumerate}
\item[(1)] ${\Bbb D}_i$ is acyclic, that is, $\widetilde{H}_j({\Bbb D}_i)=0$ for any $j$.
\item[(2)] There exists a unique $\kappa_i$ such that $\widetilde{H}_{\kappa_i}({\Bbb S}_i)=\Z$ and
 $\widetilde{H}_j({\Bbb S}_i)=0$ for any $j\neq \kappa_i$.
\end{enumerate}
Then our objective is to calculate the cohomology of
$$\mathcal{Z}_K^{(\underline{\Bbb D}, \underline{\Bbb S})}:=K(\underline{\Bbb D}, \underline{\Bbb S})=
\bigcup_{\sigma\in K}B_\sigma(\underline{\Bbb D}, \underline{\Bbb S})\subset \prod_{i=1}^m
{\Bbb D}_i.$$

First, for each $i\in [m]$, it follows immediately from  the long exact sequence of $({\Bbb D}_i, {\Bbb S}_i)$ that
$$\begin{CD}
0=\widetilde{H}^{\kappa_i}({\Bbb D}_i;{\bf k}) \longrightarrow \widetilde{H}^{\kappa_i}({\Bbb S}_i;{\bf k})\stackrel{\cong}{\longrightarrow}\widetilde{H}^{\kappa_i+1}(({\Bbb
D}_i, {\Bbb S}_i; {\bf k})\longrightarrow \widetilde{H}^{\kappa_i+1}({\Bbb
D}_i; {\bf k})=0.
 \end{CD}$$
On
the cellular cochain level, one has the following short exact sequence
$$\begin{CD}
0\longrightarrow D^*({\Bbb D}_i, {\Bbb S}_i;{\bf k}) \stackrel{j^*}{\longrightarrow} D^*({\Bbb D}_i;{\bf k})\stackrel{i^*}{\longrightarrow} D^*({\Bbb S}_i; {\bf k})\longrightarrow 0
 \end{CD}$$
 where each $D^k({\Bbb D}_i, {\Bbb S}_i;
{\bf k})$ can be considered as a subgroup of $D^k({\Bbb D}_i;
{\bf k})$, so $j^*$ is an inclusion.
By the zig-zag lemma, one can choose a  $\kappa_i$-cochain $x_i$ of $D^{\kappa_i}({\Bbb D}_i; {\bf k})$ such that
\begin{enumerate}
\item[$\bullet$] $i^*(x_i)$ represents a generator of $\widetilde{H}^{\kappa_i}({\Bbb S}_i;{\bf k})$.
\item[$\bullet$] $dx_i\in \ker i^*$ so $j^*(dx_i)=dx_i\in D^{\kappa_i+1}({\Bbb D}_i, {\Bbb S}_i;
{\bf k})\subseteq D^{\kappa_i+1}({\Bbb D}_i;
{\bf k})$ generates $\widetilde{H}^{\kappa_i+1}(({\Bbb
D}_i, {\Bbb S}_i; {\bf k})$, where $d$ is the coboundary operator of $D^*({\Bbb D}_i;{\bf k})$.
\end{enumerate}
Write $x_i^{(1)}=x_i$ and $x_i^{(2)}=dx_i$, and let $x_i^{(0)}$ denote the constant 0-cochain 1 in $D^0({\Bbb D}_i; {\bf k})$. Obviously, $x_i^{(0)}$, $x_i^{(1)}$ and $x_i^{(2)}$ are linearly independent in $D^*({\Bbb D}_i;{\bf k})$ as a
${\bf k}$-vector space.

\vskip .2cm Now let us work in the  cellular cochain complex $D^*(\prod_{i=1}^m{\Bbb D}_i; {\bf k})$ of the
 product space $\prod_{i=1}^m{\Bbb D}_i$. Let $\Omega^*$ be the vector subspace of  $D^*(\prod_{i=1}^m{\Bbb D}_i; {\bf k})$
 spanned by the following cross products
 $$x_1^{(k_1)}\times\cdots\times x_m^{(k_m)}, \ k_i\in \{0, 1, 2\}.$$
An easy observation shows that $\Omega^*$ is a cochain subcomplex of $D^*(\prod_{i=1}^m{\Bbb D}_i; {\bf k})$, and
$\big\{x_1^{(k_1)}\times\cdots\times x_m^{(k_m)}\big| \ k_i\in \{0, 1, 2\}\big\}$ forms a basis of $\Omega^*$ as a vector space  over ${\bf k}$
since $x_i^{(0)}$, $x_i^{(1)}$ and $x_i^{(2)}$ are linearly independent in $D^*({\Bbb D}_i;{\bf k})$.
For a convenience, we write each basis element $x_1^{(k_1)}\times\cdots\times x_m^{(k_m)}$ of $\Omega^*$ as the following form
$${\bf x}^{(\tau, \sigma)}$$
where ${\bf x}=x_1^{(k_1)}, ..., x_m^{(k_m)}$, $\tau=\{i \big|\ k_i=1\}$ and $\sigma=\{i\big|\ k_i=2\}$. In particular, if $\tau=\sigma=\emptyset$,
then  ${\bf x}^{(\emptyset, \emptyset)}=x_1^{(0)}\times\cdots \times x_m^{(0)}$. Thus, $\Omega^*$ can be
expressed as
$$\Omega^*=\mathrm{Span}\big\{{\bf x}^{(\tau, \sigma)}\big| \tau, \sigma\subseteq [m] \text{ with } \tau\cap \sigma=\emptyset\big\}.$$
Next by $\Phi_K$ we denote the composition
$$\Omega^*\hookrightarrow D^*(\prod_{i=1}^m{\Bbb D}_i; {\bf k})\stackrel{l^*}{\longrightarrow}
D^*(\mathcal{Z}_K^{(\underline{{\Bbb D}},\underline{{\Bbb S}})}; {\bf k})$$
where the latter map $l^*$ is induced by the inclusion $l: \mathcal{Z}_K^{(\underline{{\Bbb D}},\underline{{\Bbb S}})}\hookrightarrow \prod_{i=1}^m{\Bbb D}_i$, and it is surjective. Set
$$S_K=\mathrm{Span}\big\{{\bf x}^{(\tau, \sigma)}\in \Omega^*\big| \sigma\not\in K\big\}.$$
Clearly it is a cochain subcomplex of $\Omega^*$.

\begin{lem}
$S_K\subseteq \ker \Phi_K$. Furthermore, $\Phi_K$ can induce a cochain map $\Omega^*/S_K\longrightarrow
 D^*(\mathcal{Z}_K^{(\underline{{\Bbb D}},\underline{{\Bbb S}})}; {\bf k})$, also denoted by $\Phi_K$.
\end{lem}

\begin{proof}
Let ${\bf x}^{(\tau, \sigma)}$ be a basis element in $S_K\subset D^*(\prod_{i=1}^m{\Bbb D}_i; {\bf k})$. For any product cell $e=e_1\times\cdots\times e_m\subset \mathcal{Z}_K^{(\underline{{\Bbb D}},\underline{{\Bbb S}})}\subseteq\prod_{i=1}^m{\Bbb D}_i$, there must be some $\sigma'\in K$ such that
$e\subset B_{\sigma'}(\underline{{\Bbb D}},\underline{{\Bbb S}})$, where each $e_i$ can represent a generator in the celluar chain group
$D_{\dim e_i}({\Bbb D}_i; {\bf k})$. In addition, it is easy to see that $e$ can also be regarded as a generator of
the cellular chain complex $D_*(\mathcal{Z}_K^{(\underline{{\Bbb D}},\underline{{\Bbb S}})}; {\bf k})
\stackrel{l_*}{\hookrightarrow} D_*(\prod_{i=1}^m{\Bbb D}_i; {\bf k})$ where $l_*$ is the inclusion
induced by $l: \mathcal{Z}_K^{(\underline{{\Bbb D}},\underline{{\Bbb S}})}\hookrightarrow\prod_{i=1}^m{\Bbb D}_i$. Since $\sigma\not\in K$, $\sigma$ is non-empty. Moreover, there is some
$i_0\in \sigma\setminus\sigma'$ such that $e_{i_0}\subset {\Bbb S}_{i_0}\subset{\Bbb D}_{i_0}$ and the factor $x_{i_0}^{(2)}\in D^{\kappa_{i_0}+1}({\Bbb D}_{i_0}, {\Bbb S}_{i_0};{\bf k})\subset D^{\kappa_{i_0}+1}({\Bbb D}_{i_0};{\bf k})$ in ${\bf x}^{(\tau, \sigma)}$, together yielding that
$\langle x_{i_0}^{(2)}, e_{i_0}\rangle=0$. 
Therefore, $\langle{\bf x}^{(\tau, \sigma)}, l_*(e)\rangle=\langle{\bf x}^{(\tau, \sigma)}, e\rangle=0$ by the definition of cross product.
Furthermore,  we  have that the value of $\Phi_K({\bf x}^{(\tau, \sigma)})$ on $e$ is
 $$\langle\Phi_K({\bf x}^{(\tau, \sigma)}), e\rangle=\langle{\bf x}^{(\tau, \sigma)}\circ l_*, \ e\rangle=
 \langle{\bf x}^{(\tau, \sigma)}, l_*(e)\rangle=0$$
 so $\Phi_K({\bf x}^{(\tau, \sigma)})=0$
 in $D^*(\mathcal{Z}_K^{(\underline{{\Bbb D}},\underline{{\Bbb S}})}; {\bf k})$, as desired.
\end{proof}

By $\Omega^*(K)$ we denote the quotient $\Omega^*/S_K$. Let $L$ be a subcomplex of $K$.  Then we obtain a pair
$(\mathcal{Z}_K^{(\underline{{\Bbb D}},\underline{{\Bbb S}})}, \mathcal{Z}_L^{(\underline{{\Bbb D}},\underline{{\Bbb S}})})$ of
CW-complexes. Now since $S_K\subseteq S_L$, we have a short exact sequence
\begin{equation}\label{short exact}
0 \longrightarrow\ker(\pi^*)\longrightarrow \Omega^*(K) \stackrel{\pi^*}{\longrightarrow}
\Omega^*(L)\longrightarrow 0
 \end{equation}
 where $\pi^*$ is induced by the natural inclusion $\pi: S_K\hookrightarrow S_L$. By $\Omega^*(K, L)$ we denote the kernel $\ker\pi^*$. It is easy to see that two cochain maps $\Phi_K: \Omega^*(K)\longrightarrow D^*(\mathcal{Z}_K^{(\underline{{\Bbb D}},\underline{{\Bbb S}})}; {\bf k})$ and $\Phi_L: \Omega^*(L)\longrightarrow D^*(\mathcal{Z}_L^{(\underline{{\Bbb D}},\underline{{\Bbb S}})}; {\bf k})$
 give a cochain map $\Phi_{(K,L)}: \Omega^*(K,L)\longrightarrow D^*(\mathcal{Z}_K^{(\underline{{\Bbb D}},\underline{{\Bbb S}})}, \mathcal{Z}_L^{(\underline{{\Bbb D}},\underline{{\Bbb S}})}; {\bf k})$ such that the following diagram commutes
$$ \xymatrix{
 0 \ar[r]&  \Omega^*(K,L) \ar[d]_{\Phi_{(K,L)}} \ar[r] & \Omega^*(K) \ar[d]_{\Phi_K} \ar[r]^{\pi^*}
                & \Omega^*(L) \ar[d]^{\Phi_L} \ar[r] & 0 \\
 0 \ar[r] & D^*(\mathcal{Z}_K^{(\underline{{\Bbb D}},\underline{{\Bbb S}})}, \mathcal{Z}_L^{(\underline{{\Bbb D}},\underline{{\Bbb S}})}; {\bf k})  \ar[r]
                & D^*(\mathcal{Z}_K^{(\underline{{\Bbb D}},\underline{{\Bbb S}})}; {\bf k})  \ar[r] &
             D^*(\mathcal{Z}_L^{(\underline{{\Bbb D}},\underline{{\Bbb S}})}; {\bf k})    \ar[r] & 0.      }$$
Furthermore, we may obtain a homomorphism between two long exact cohomology sequences given by two short exact sequences above.

\begin{prop}\label{module-iso}
For any $K\in \mathcal{K}_{[m]}$, $\Phi_K$ induces an isomorphism
$$H^*(\Omega^*(K); {\bf k})\stackrel{\cong}{\longrightarrow}
H^*(\mathcal{Z}_K^{(\underline{{\Bbb D}}, \underline{{\Bbb S}})}; {\bf k})$$
as graded ${\bf k}$-modules.
\end{prop}

\begin{proof}
First observe that for $K=\{\emptyset\}$, $\Omega^*(K)$ is spanned by
$\{{\bf x}^{(\tau, \emptyset)}\big| \tau\subseteq [m]\}$ with zero coboundary operator. On the other hand, if $K=\{\emptyset\}$ then  $\mathcal{Z}_K^{(\underline{{\Bbb D}}, \underline{{\Bbb S}})}=\prod_{i=1}^m{\Bbb S}_i$. By the K\"unneth formula, the above set is not but a basis of  $H^*(\mathcal{Z}_K^{(\underline{{\Bbb D}}, \underline{{\Bbb S}})}; {\bf k})$ as a graded ${\bf k}$-module (if we view the elements of the set as cohomological classes).
Thus, clearly $\Phi_K$ induces an isomorphism in this case.

\vskip .2cm
Next we proceed inductively by considering a pair of abstract simplicial complexes $(K, L)$ where $K=L\sqcup \{\sigma_0\}$ for some simplex $\sigma_0$ (which is a maximal element of $K$ as a poset). Hence $(\mathcal{Z}_K^{(\underline{{\Bbb D}}, \underline{{\Bbb S}})}, \mathcal{Z}_L^{(\underline{{\Bbb D}}, \underline{{\Bbb S}})})$ is a pair of CW-complexes, which has by excision the same cohomology
as $(\mathcal{Z}_{2^{\sigma_0}}^{(\underline{{\Bbb D}}, \underline{{\Bbb S}})}, \mathcal{Z}_{2^{\sigma_0}\setminus \sigma_0}^{(\underline{{\Bbb D}}, \underline{{\Bbb S}})})$. This pair $(\mathcal{Z}_{2^{\sigma_0}}^{(\underline{{\Bbb D}}, \underline{{\Bbb S}})}, \mathcal{Z}_{2^{\sigma_0}\setminus \sigma_0}^{(\underline{{\Bbb D}}, \underline{{\Bbb S}})})$ is in turn homeomorphic to
$$\prod_{i\in [m]\setminus\sigma_0}{\Bbb S}_i\times\Big(\prod_{i\in \sigma_0}{\Bbb D}_i, A(\prod_{i\in \sigma_0}{\Bbb D}_i)\Big)$$
where $A(\prod_{i\in \sigma_0}{\Bbb D}_i)=({\Bbb S}_{i_1}\times{\Bbb D}_{i_2}\times\cdots\times {\Bbb D}_{i_s})\cup\cdots \cup
({\Bbb D}_{i_1}\times\cdots\times{\Bbb D}_{i_{s-1}}\times{\Bbb S}_{i_s})$ with $\sigma_0=\{i_1, ..., i_s\big| i_1<\cdots<i_s\}$.
By relative K\"unneth formula,  its cohomology with ${\bf k}$ coefficients is isomorphic to $$\mathrm{Span}\{{\bf x}^{(\tau, \sigma_0)}\big| \tau\subseteq [m]
 \text{ with } \tau\cap \sigma_0=\emptyset\}$$
as graded ${\bf k}$-modules.
On the other hand, we see easily from the  short exact sequence (\ref{short exact})
 that $\Omega^*(K, L)=\ker\pi^*$ is exactly equal to the cochain complex $$\mathrm{Span}\{{\bf x}^{(\tau, \sigma_0)}\big| \tau\subseteq [m]
 \text{ with } \tau\cap \sigma_0=\emptyset\}$$ with zero coboundary operator. It then follows that
 $\Phi_{(K,L)}$ induces an isomorphism $H^*(\Omega^*(K,L);{\bf k})\stackrel{\cong}{\longrightarrow}
 H^*(\mathcal{Z}_K^{(\underline{{\Bbb D}},\underline{{\Bbb S}})}, \mathcal{Z}_L^{(\underline{{\Bbb D}},\underline{{\Bbb S}})}; {\bf k})$ as graded ${\bf k}$-modules.
   Inductively, now we may assume that $\Phi_L$ induces an isomorphism $H^*(\Omega^*(L);{\bf k})\longrightarrow
H^*(\mathcal{Z}_L^{(\underline{{\Bbb D}}, \underline{{\Bbb S}})};{\bf k})$ as graded ${\bf k}$-modules. Hence we may conclude that the same holds for $H^*(\Omega^*(K);{\bf k})\longrightarrow
H^*(\mathcal{Z}_K^{(\underline{{\Bbb D}}, \underline{{\Bbb S}})};{\bf k})$ by the five-lemma. This completes the induction and the proof of Proposition~\ref{module-iso}.
\end{proof}

Now let us return back to study the complex $(\Omega^*(K), \underline{d})$.
    First, we may impose a $\{0,1\}^m$-graded (or $2^{[m]}$-graded) structure on $\Omega^*(K)$, by defining
    for $a\in 2^{[m]}$
\begin{equation*}
\Omega^*(K)_a:
=\text{Span}\big\{{\bf x}^{(\tau, \sigma)}\big| \tau\subseteq [m], \sigma\in K \text{ with }
\tau\cup \sigma=a, \tau\cap\sigma=\emptyset\big\}.
\end{equation*}
 Then, clearly $\Omega^*(K)=\bigoplus_{a\in [m]}\Omega^*(K)_a$. Furthermore, given a basis element ${\bf x}^{(\tau, \sigma)}
\in \Omega^*(K)_a$ with $\tau=a\setminus\sigma$, by a direct calculation we have that
\begin{equation*}\label{cob-op}
\underline{d}({\bf x}^{(a\setminus\sigma, \sigma)})=\sum_{k\in a\setminus\sigma \atop \sigma\cup\{k\}\in K}\epsilon_k{\bf x}^{(a\setminus(\sigma\cup\{k\}),\sigma\cup\{k\})}
\end{equation*}
which still belongs to  $\Omega^*(K)_a$, where $\epsilon_k=\pm1$. So $(\Omega^*(K)_a, \underline{d})$ has also a cochain complex structure.
This means that $\Omega^*(K)$ is a bigraded ${\bf k}$-module. Also, clearly the basis of  $\Omega^*(K)_a$ is indexed by $K|_a$
 where $K|_a=\{\sigma\in K\big| \sigma\subseteq a\}$.

\begin{lem}\label{hochster}
For each $a\in 2^{[m]}$, $(\Omega^*(K)_a, \underline{d})$ is isomorphic to
the coaugmented cochain complex
$(C^*(K|_a; {\bf k}), d')$ as cochain complexes.
Furthermore, $H^*(\Omega^*(K)_a; {\bf k})$ $\cong \widetilde{H}^*(K|_a; {\bf k})$ as graded ${\bf k}$-modules.
\end{lem}

Lemma~\ref{hochster} is a (dualized) consequence of the following general result.

\begin{lem}
Let $K$ be an abstract simplicial complex on a finite set.  Let $V(K)$ be a vector
space over ${\bf k}$ with a $K$-indexed basis $\{v_\sigma| \sigma\in K\}$, and let $\iota: V(K)\longrightarrow V(K)$ be a linear map such that $\iota^2= 0$ and $\iota(v_\sigma)=\sum_{k\in \sigma}\varepsilon_k v_{\sigma\setminus\{k\}}$
 where $\varepsilon_k=\pm1$. Then there is an isomorphism $f: V(K)\longrightarrow  C_*(K;{\bf k})$ as ${\bf k}$-vector spaces with form $f: v_\sigma \longmapsto \varepsilon_\sigma\sigma$ such that $f\circ\iota= \partial\circ f$, where $\varepsilon_\sigma=\pm1$ and $C_*(K;{\bf k})$ is the ordinary chain complex over ${\bf k}$ of $K$ with the boundary operator $\partial$.
\end{lem}

\begin{proof} We proceed inductively.
For $K=\{\emptyset\}$, $V(K)=\text{Span}\{v_\emptyset\}\cong {\bf k}$ with $\iota=0$
and $C_*(K; {\bf k})=\text{Span}\{\emptyset\}\cong{\bf k}$ with $\partial=0$, so clearly we have such a $f$.
 Now for an arbitrary $K\not=\{\emptyset\}$, take a maximal element $\sigma_0$ of $K$ (as a poset) so that $L=K\setminus\{\sigma_0\}$ is a subcomplex of $K$. The subspace $V(K)|_L=\text{Span}\{v_\sigma|\sigma\in L\}$
 is invariant under  $\iota$. So we can apply induction hypothesis to
 $(V(K)|_L, \iota)$, yielding an isomorphism $f_0: V(K)|_L\longrightarrow
C_*(L;{\bf k})$ by $v_\sigma\longmapsto \varepsilon_\sigma\sigma$ such that $f_0\circ \iota=\partial \circ f_0$.
 Now observe that $\iota(v_{\sigma_0})=\sum_{k\in \sigma_0} \varepsilon_k v_{\sigma_0\setminus\{k\}}
 \in V(K)|_L$, so $f_0(\iota(v_{\sigma_0}))= \sum_{k\in \sigma_0} \varepsilon_k \varepsilon_{\sigma_0\setminus\{k\}} (\sigma_0\setminus\{k\})$ which is in the chain group $C_{|\sigma_0|-2}(2^{\sigma_0};{\bf k})\subset C_{|\sigma_0|-2}(L; {\bf k})$, and
 $(\partial\circ f_0)(\iota(v_{\sigma_0}))=(f_0\circ\iota)(\iota(v_{\sigma_0}))=f_0(\iota^2(v_{\sigma_0}))=0$, i.e.,
 $f_0(\iota(v_{\sigma_0}))\in \ker\partial$.
Since $C_*(2^{\sigma_0};{\bf k})$ is acyclic and $C_{|\sigma_0|-1}(2^{\sigma_0};{\bf k})=\text{Span}\{
 \sigma_0\}$, we have $f_0(\iota(v_{\sigma_0}))=\partial (n\sigma_0)$ for some $n\in {\bf k}$.
 However, $\partial (n\sigma_0)=n\partial(\sigma_0)$ so  $n\partial(\sigma_0)=\sum_{k\in \sigma_0} \varepsilon_k \varepsilon_{\sigma_0\setminus\{k\}}(\sigma_0\setminus\{k\})$. This forces $n$ to be $\pm1$.
 We can then extend $f_0$ to $f: V(K)\longrightarrow C_*(K;{\bf k})$ by defining $v_{\sigma_0}\longmapsto
 n\sigma_0$, so that we have
 $$f(\iota(v_{\sigma_0}))=f_0(\iota(v_{\sigma_0}))=\partial(n\sigma_0)=\partial(f(v_{\sigma_0})).$$
 Hence $f\circ \iota=\partial \circ f$ in $V(K)$.  The induction step is finished, proving the lemma.
\end{proof}

The famous  Hochster formula tells us (see \cite[Corollary 5.12]{ms}) that for each $a\in 2^{[m]}$,
$$\widetilde{H}^{|a|-i-1}(K|_a; {\bf k})\cong \text{Tor}^{{\bf k}[{\bf v}]}_i({\bf k}(K), {\bf k})_a.$$
We know by Lemma~\ref{hochster} that each class of $\widetilde{H}^{|a|-i-1}(K|_a; {\bf k})$ may be understood as one of $H^*(\Omega^*(K)_a; {\bf k})$, represented by a linear combination of the elements of the form ${\bf x}^{(a\setminus\sigma, \sigma)}\in \Omega^*(K)_a$ with $|\sigma|=|a|-i$; so
by Proposition~\ref{module-iso} it corresponds to a cohomological class of degree
$|\sigma|+\sum_{k\in a}\kappa_k=-i+\sum_{k\in a}(\kappa_k+1)$ in
$H^*(\mathcal{Z}_K^{(\underline{{\Bbb D}}, \underline{{\Bbb S}})};{\bf k})$. To sum up, it follows that for each $n\geq 0$,
$$H^n(\mathcal{Z}_K^{(\underline{{\Bbb D}}, \underline{{\Bbb S}})};{\bf k})\cong \bigoplus_{a\in 2^{[m]} \atop -i+\sum_{k\in a}(\kappa_k+1)=n}\text{\rm Tor}_i^{{\bf k}[{\bf v}]}({\bf k}(K), {\bf k})_a.$$
Combining  with all arguments above, we conclude  that
\begin{thm}\label{module-str}
As graded ${\bf k}$-modules,
$$H^*(\mathcal{Z}_K^{(\underline{{\Bbb D}}, \underline{{\Bbb S}})};{\bf k})\cong \text{\rm Tor}^{{\bf k}[{\bf v}]}({\bf k}(K), {\bf k}).$$
\end{thm}

Together with Proposition~\ref{tor-betti} and Theorem~\ref{module-str}, we obtain that

\begin{cor}\label{sum-betti}
$\sum\limits_i\dim_{\bf k}H^i(\mathcal{Z}_{K}^{(\underline{{\Bbb D}}, \underline{{\Bbb S}})}; {\bf k})=\sum\limits_{i=0}^h\sum\limits_{a\in
2^{[m]}}\beta_{i,a}^{{\bf k}(K)}.$
\end{cor}

\begin{rem}\label{ring str} It should be pointed out that here we merely determine the ${\bf k}$-module structure of $H^*(\mathcal{Z}_K^{(\underline{{\Bbb D}}, \underline{{\Bbb S}})};{\bf k})$. Of course, this is enough for our purpose
in this paper.
Observe that if there are two  $i, j\in [m]$ with $i\not=j$ such that $\kappa_i$ and $\kappa_j$ are even, then
for $x_i^{(2)}, x_j^{(2)}\in \Omega^*(K)$, $x_i^{(2)}\times x_j^{(2)}=-x_j^{(2)}\times x_i^{(2)}$. This means that in this case,
if ${\bf k}$ is not a field of characteristic 2, then $H^*(\Omega^*(K); {\bf k})$ cannot be isomorphic to $\text{\rm Tor}^{{\bf k}[{\bf v}]}({\bf k}(K), {\bf k})$ as ${\bf k}$-algebras since ${\bf k}(K)$ is a commutative ring.
Even when $\bf k$ is a field of characteristic 2, there is still
some nuance preventing us from simply extending the ring structure
result (\ref{bp}) of Buchstaber and Panov  to the case of, say
$\R\mathcal{Z}_K$; Indeed, in this case $x_i^{(1)}$ would be a $0$-cochain,
which satisfies $x_i^{(1)}\cup x_i^{(1)}=x_i^{(1)}$, whereas in the
cases when $\kappa_i>0$, $x_i^{(1)}\cup x_i^{(1)}$ would be instead
zero element in $H^*({\Bbb S}_i; {\bf k})$. Nevertheless, our calculation
of the module structure actually represents any cohomological class
in $H^*(\mathcal{Z}_{K}^{(\underline{{\Bbb D}}, \underline{{\Bbb
S}})}; {\bf k})$ as a sum of ${\bf x}^{(\tau, \sigma)}$'s via the isomorphism
$H^*(\Omega^*(K);{\bf k})\cong H^*(\mathcal{Z}_{K}^{(\underline{{\Bbb D}},
\underline{{\Bbb S}})};{\bf k})$, from which we may also figure out the
cohomological equivalence relation amongst such sums; since the cup
product of pairs of these elements is clear, in a certain sense we should
have also determined the ring structure of
$H^*(\mathcal{Z}_{K}^{(\underline{{\Bbb D}}, \underline{{\Bbb
S}})};{\bf k})$. In other words, let ${\bf k}(K)={\bf k}[{\bf v}]/I_K={\bf k}[v_1, ...,v_m]/I_K$ be the  Stanley--Reisner face ring of $K$ with
$\deg v_i=\kappa_i+1$. Then it should be reasonable to {\em conjecture} that the following  results  hold:
\begin{enumerate}
\item[$\bullet$] If all $\kappa_i$'s are odd, then $H^*(\mathcal{Z}_K^{(\underline{{\Bbb D}},
\underline{{\Bbb S}})};{\bf k})\cong \text{\rm Tor}^{{\bf k}[{\bf v}]}({\bf k}(K), {\bf k})$ as ${\bf k}$-algebras.
\item[$\bullet$] If  $\kappa_i>0$ for any $i\in [m]$, then $H^*(\mathcal{Z}_K^{(\underline{{\Bbb D}}, \underline{{\Bbb S}})};{\bf k}_2)\cong \text{\rm Tor}^{{\bf k}_2[{\bf v}]}({\bf k}_2(K), {\bf k}_2)$ as ${\bf k}_2$-algebras.
   \item[$\bullet$]  In general,  $H^*(\mathcal{Z}_K^{(\underline{{\Bbb D}}, \underline{{\Bbb S}})};{\bf k}_2)\cong H[H^*(\prod_{i=1}^m{\Bbb S}_i;{\bf k}_2)\otimes_{{\bf k}_2[{\bf v}]}{\bf k}_2(K)]$ as ${\bf k}_2$-algebras.
\end{enumerate}
\end{rem}

\section{Application to the free actions on $\mathcal{Z}_K$ and ${\Bbb R}\mathcal{Z}_K$}\label{s5}

First we prove a useful lemma.

\begin{lem}\label{lem:coloring}
Let $K\in \mathcal{K}_{[m]}$ be an abstract simplicial complex on
vertex set $[m]$, and let $H$ $($resp. $H_{\Bbb R})$ be a rank $r$
subtorus of $T^m$ $($resp. $(\Z_2)^m)$. If $H$ $($resp. $H_{\Bbb
R})$ can freely act on $\mathcal{Z}_K$ $($resp. ${\Bbb
R}\mathcal{Z}_K)$, then $r\leq m-\dim K-1$.
\end{lem}

\begin{proof}
 It is well-known
that $H$ (resp. $H_{\Bbb R}$) can freely act on $\mathcal{Z}_K$
(resp. ${\Bbb R}\mathcal{Z}_K$) if and only if for any point $z$
(resp. $x$) of $\mathcal{Z}_K$ (resp. ${\Bbb R}\mathcal{Z}_K$),
$H\cap G_z$ (resp. $H_{\Bbb R}\cap G_x$) is trivial, where $G_z$
(resp. $G_x$) is the isotropy subgroup at $z$ (resp. $x$). Suppose
that $r>m-\dim K-1$. Take $a\in K$ with $|a|=\dim K+1$. Without the
loss of generality, assume that $a=\{1, ..., |a|\}$. Then we see
that $\mathcal{Z}_K$ (resp. ${\Bbb R}\mathcal{Z}_K$) contains  the
point of the form $z=(0, ..., 0, z_{|a|+1}, ..., z_m)$ (resp. $x=(0,
..., 0, x_{|a|+1}, ..., x_m)$). It is easy to see that the isotropy
subgroup $G_z$ (resp. $G_x$) has rank at least $|a|$, so the
intersection $H\cap G_z$ (resp. $H_{\Bbb R}\cap G_x$) cannot be
trivial. This contradiction means that $r$ must be equal to or less
than $m-\dim K-1$.
\end{proof}

Now let us use the preceding results to complete the proof of
Theorem~\ref{main}.

\vskip .2cm

 \noindent {\em Proof Theorem~\ref{main}}. Let
$f=\sum_{a\in K}\delta_a\in\mathcal{F}_{[m]}$ such that
$\mathrm{supp}(f)=K$. If $f=\underline{1}$ (i.e., $K=2^{[m]}$), then
$\mathcal{Z}_K=(D^2)^m$ (resp. ${\Bbb R}\mathcal{Z}_K=(D^1)^m$).
However, any properly nontrivial subtorus of $T^m$  (resp.
$(\Z_2)^m$) cannot freely act on $(D^2)^m$ (resp. $(D^1)^m$) since
the point $(0, ..., 0)$ is always a fixed point. Thus we may assume
that $f\not=\underline{1}$.   By Theorem \ref{t:extend}, there exists some $a\in 2^{[m]}$ with
$a\not=[m]$ such that $a\in \mathrm{supp}(f)=K$ and
$|\mathrm{supp}(\Mob(f))|\geq 2^{m-n}$ where $n=|a|$. Since $a\in
K$, we have that $n\leq \dim K+1$. So by Lemma \ref{lem:coloring} it
follows that $n\leq m-r$ and  $r\leq m-n$. Combining  with
Theorem~\ref{t:extend} and Corollaries~\ref{bt} and~\ref{sum-betti}  together
gives
$$2^r\leq 2^{m-n}\leq|\mathrm{supp}(\Mob(f))|\leq \sum_i\dim_{{\bf k}}H^i(\mathcal{Z}_K; {\bf k})=\sum_i\dim_{{\bf k}}H^i({\Bbb R}\mathcal{Z}_K; {\bf k})$$
as desired. \hfill $\Box$

\vskip .2cm

\begin{acknow}
The authors are grateful to M. Franz, M. Masuda, T. E. Panov, V. Puppe and L. Yu for their comments and suggestions.
\end{acknow}


\begin{thebibliography}{99}
\bibitem{ab} A. Adem and W. Browder, {\em The free rank of symmetry
of $(S^n)^k$}, Invent. Math. {\bf 92} (1988), 431--440.
\bibitem{ad} A. Adem and J. F. Davis, {\em Topics in transformation
groups}, Handbook of Geometric Topology, 1--54, North--Holland,
Amsterdam, 2002.
\bibitem{ap} C. Allday and V. Puppe, {\em Cohomological methods
in transformation groups}, Cambridge Studies in Advanced
Mathematics, {\bf 32}, Cambridge University Press,  1993.
\bibitem{bbcg} A. Bahri, M. Bendersky, F. R. Cohen and S.
Gitler, {\em Decompositions of the polyhedral product functor with
applications to moment-angle complexes and related spaces}, Proc.
Nat. Acad. Sci. U.S.A. {\bf 106} (2009),  12241--12244.
\bibitem{bp} V. M. Buchstaber and T. E. Panov, {\em Torus actions and their applications in topology and
combinatorics},
 University Lecture Series, Vol. {\bf 24}, Amer.
Math. Soc., Providence, RI, 2002.
\bibitem{bp1} V. M. Buchstaber and T. E. Panov,  {\em Combinatorics of simplicial cell complexes and
torus action},  Proc. Steklov Inst. Math. {\bf 247} (2004), 33--49.
 \bibitem{ca1} G. Carlsson, {\em On the non-existence of free
 actions of elementary abelian groups on products of spheres}, Amer.
 J. Math. {\bf 102} (1980), 1147--1157.
 \bibitem{ca2} G. Carlsson, {\em On the rank  of abelian groups
 acting  freely on $(S^n)^k$}, Invent. Math.  {\bf 69} (1982), 393--400.
 \bibitem{ca3} G. Carlsson, {\em Free $(\Z/2)^k$-actions and a problem in commutative algebra},
 Transformation groups, Pozna\'n 1985, 79--83, Lecture Notes in Math., {\bf 1217}, Springer, Berlin, 1986.
 \bibitem{c} P. E. Conner, {\em On the action of a finite group on $S^n\times S^n$},  Ann. of Math. {\bf 66} (1957),  586--588.
\bibitem{dj} M. W.  Davis and T. Januszkiewicz, {\em Convex polytopes, Coxeter orbifolds and torus actions},
Duke Math. J. {\bf 62} (1991), 417--451.
\bibitem{ds} G. Denham and A. Suciu, {\em Moment-angle complexes, monomial ideals and Massey
products},  Pure Appl. Math.  Q.  {\bf 3} (2007), 25--60.
\bibitem{h} S. Halperin, {\em Rational homotopy and torus actions}. Aspects of topology, 293--306,
London Math. Soc. Lecture Note Ser., {\bf 93}, Cambridge Univ.
Press, Cambridge, 1985.
\bibitem{ms} E. Miller and  B. Sturmfels, {\em Combinatorial Commutative
Algebra}, Graduate Texts in Math. {\bf 227}, Springer, 2005.
\bibitem{pa} T. E. Panov, {\em Cohomology of face rings and torus actions},  London Math. Soc. Lect.
Note Ser. {\bf 347} (2008), 165--201.
\bibitem{p} V. Puppe, {\em Multiplicative aspects of the Halperin--Carlsson
conjecture}, arXiv:0811.3517
\bibitem{s} P. A. Smith, {\em Orbit spaces of abelian $p$-groups},  Proc. Nat. Acad. Sci. U.S.A. {\bf 45} (1959), 1772--1775.
\bibitem{y} E. Yalcin, {\em Groups actions and group extensions}, Trans. Amer. Math. Soc. {\bf 352} (2000), 2689--2700.
 \end{thebibliography}
\end{document}